\documentclass[a4paper,11pt]{article}
\usepackage{amsmath,textcomp,latexsym,graphics,graphicx}
\usepackage{amsfonts}
\usepackage{amssymb,amsthm}

\DeclareGraphicsExtensions{.pdf,.png,.jpg}

\def\dis{\displaystyle}
\def\ep{\varepsilon}
\def\p{\partial}
\def\Om{\Omega}
\def\om{\omega}
\def\R{\mathbb{R}}
\def\N{\mathbb{N}}

\def\U{{\mathcal U}}
\def\A{{\mathcal A}}
\def\Rbar{{\overline\R}}
\def\cp{\mathrm{cap}\,}

\newcommand{\Hc}{\mathcal H}

\newcommand{\sm}{\setminus}

\newcommand{\Lin}{{{\mathcal L}}}

\newcommand{\g}{\gamma}

\newcommand{\vphi}{\varphi}

\newcommand{\lra}{\longrightarrow}
\newcommand{\Lra}{\Longrightarrow}

\newcommand{\sr}{\stackrel}

\newcommand{\nif}{{n \rightarrow \infty}}

\newtheorem{theo}{Theorem}[section]
\newtheorem{prop}[theo]{Proposition}
\newtheorem{lem}[theo]{Lemma}

\newtheorem{defi}[theo]{Definition}
\newtheorem{rem}[theo]{Remark}

\title{Minimization of $\lambda_2(\Omega)$ with a perimeter constraint}
\author{Dorin Bucur\\
Laboratoire de Math\'ematiques (LAMA)\\ 
UMR 5127, Universit\'e de Savoie\\
Campus Scientifique 73 376 Le-Bourget-Du-Lac, France \\
email: dorin.bucur@univ-savoie.fr
 \and 
Giuseppe Buttazzo\\
Dipartimento di Matematica\\
Universita di Pisa, Largo B.Pontecorvo 5\\
56127 Pisa, Italy \\
{buttazzo@dm.unipi.it}
\and
Antoine Henrot\\ Institut \'Elie Cartan Nancy\\ 
UMR 7502, Nancy Universit\'e - CNRS - INRIA\\
B.P. 239 54506 Vandoeuvre les Nancy Cedex, France\\
email: henrot@iecn.u-nancy.fr}

\begin{document}

\maketitle

\begin{abstract}
We study the problem of minimizing the second Dirichlet eigenvalue for the Laplacian operator among sets of given perimeter. In two dimensions, we prove that the optimum exists, is convex, regular, and its boundary contains exactly two points where the curvature vanishes. In $N$ dimensions, we prove a more general existence theorem for a class of functionals which is decreasing with respect to set inclusion and $\gamma$ lower semicontinuous.
\end{abstract}

\vspace{0.6cm}

\vspace{0.3cm}

\newpage

\section{Introduction}
Let $\Omega$ be a bounded open set in $\mathbb{R}^N$ and let us denote by $0<\lambda_1(\Omega)\le\lambda_2(\Omega)\le\lambda_3(\Omega)\ldots$ its eigenvalues for the Laplacian operator with homogeneous Dirichlet boundary condition. Problems linking the shape of a domain to the sequence of its eigenvalues, or to some function of them, are among the most fascinating of mathematical analysis or differential geometry. In particular, problems of minimization of eigenvalues, or combination of eigenvalues, brought about many deep works since the early part of the twentieth century. Actually, this question appears first in the famous book of Lord Rayleigh ``The theory of sound''. Thanks to some explicit computations and "physical evidence", Lord Rayleigh conjectured that the disk should minimize the first Dirichlet eigenvalue $\lambda_1$ of the Laplacian among plane open sets of given area. This result has been proved later by Faber and Krahn using a rearrangement technique. Then, many other similar "isoperimetric problems" have been considered. For a survey on these questions, we refer to the papers \cite{Ash-Ben}, \cite{BBBS}, \cite{Pay2}, \cite{Yau2} and to the recent books \cite{H}, \cite{kebook}.

Usually, in these minimization problems, one works in the class of sets with a given measure. In this paper, on the contrary we choose to look at similar problems but with a constraint on the perimeter of the competing sets. Apart the mathematical own interest of this question, the reason which led us to consider this problem is the following. Studying the famous gap problem (originally considered in \cite{vdB}, see Section 7 in \cite{Ash-Ben} for a comprehensive bibliography on this problem), we were interested in minimizing $\lambda_2(\Om)-\lambda_1(\Om)$, and more generally $\lambda_2(\Om)- k\lambda_1(\Om)$, with $0\leq k\leq 1$, among (convex) open sets of given diameter. Looking at the limiting case $k=0$, we realized that the optimal set (which does exist) is a body of constant width. Since all bodies of constant width have the same perimeter in dimension two, we were naturally led to consider the problem of minimizing $\lambda_2(\Om)$ among sets of given perimeter. In particular, if the solution was a ball (or more generally a body of
constant width), it would give the answer to the previous problem. Unfortunately, as it is shown in Theorem \ref{theo3}, it is not the case! The minimizer that we are able to identify and characterize here (at least in two dimension) is a particular regular convex body, with two points on its boundary where the curvature vanishes. It is worth observing that the four following minimization problems for the second eigenvalue have different solutions:
\begin{itemize}
\item with a volume constraint: two identical disks (see \cite{Kra2} or
\cite{H}),
\item with a volume and a convexity constraint: a stadium-like set (see
\cite{He-Ou}),
\item with a perimeter constraint: the convex set described in this paper,
\item with a diameter constraint: we conjecture that the solution is a disk.
\end{itemize}
Let us remark that the same problems for the first eigenvalue all have the disk as the solution thanks to Faber-Krahn inequality and the classical isoperimetric inequality.

This paper is organized as follows: section \ref{sec2} is devoted to the complete study of the two-dimensional problem. We first prove the existence of a minimizer and its $C^\infty$ regularity. Then, we give some other qualitative and geometric properties of the minimizer.
For that purpose, we use boundary variation (the classical Hadamard's formulae) which leads to an overdetermined boundary value problem, with
$|grad u_2|^2$ proportional to the curvature of the boundary. We use this boundary condition to prove that the boundary of the optimal domain does not contain any arc of circle and segment and that the curvature of the boundary vanishes at exactly two points. In section \ref{sec3}, we consider the problem in higher dimension; this case is much more complicated since we cannot use the trick of convexification, and actually we conjecture that optimal domains are not convex (see section \ref{further} on open problems). We first give some preliminaries on capacity and $\gamma$-convergence (we refer to the book \cite{BB} for all details), then we consider a quite general minimization problem for a class of functionals decreasing with respect to set inclusion and which are $\gamma$ lower semicontinuous. We work here with measurable sets with bounded perimeter which are included in some fixed bounded domain $D$.
As Theorem \ref{equivrel} shows,  this relaxed problem is equivalent to the initial problem. For the second eigenvalue of the Laplacian we moreover prove that we can get rid of the assumption that the sets lie in some bounded subset of $\R^N$.

\section{The two-dimensional case}\label{sec2}
\subsection{Existence, regularity}\label{sec2.1}
We want to solve the minimization problem
\begin{equation}\label{2.1}
\min\{\lambda_2(\Omega),\;\Omega\subset \R^2,\; P(\Om)\leq c\}
\end{equation}
where $\lambda_2(\Omega)$ is the second eigenvalue of the Laplacian with Dirichlet boundary condition on the bounded open set $\Om$ and $P(\Om)$ denotes the perimeter (in the sense of De Giorgi) of $\Om$. The monotonicity of the eigenvalues of the Dirichlet-Laplacian with respect to the inclusion has two easy consequences:
\begin{enumerate}
\item If $\Om^*$ denotes the convex hull of $\Om$, since {\it in two dimensions} and for a connected set, $P(\Om^*)\le P(\Om)$, it is clear that we can restrict ourselves to look for minimizers in the class of {\it convex sets} with perimeter less or equal than $c$.
\item Obviously, it is equivalent to consider the constraint $P(\Om)\le c$ or $P(\Om)= c$.
\end{enumerate}
Of course, point 1 above easily implies existence (see Theorem \ref{theo1} below), but is no longer true in higher dimension which makes the existence proof much harder, see Theorem \ref{bh07}. For the regularity of optimal domains the following lemma will be used.

\begin{lem}\label{lem1}
If $\Om$ is a minimizer of problem \eqref{2.1}, then $\lambda_2(\Om)$ is simple.
\end{lem}

\begin{proof}
The idea of the proof is to show that a double
eigenvalue would split under boundary perturbation of the domain, with one
of the eigenvalues going down. A very similar result is proved in \cite[Theorem 2.5.10]{H}. The new difficulties here are the perimeter constraint (instead of the volume) and the fact that the domain $\Omega$ is convex, but not necessarily regular. Nevertheless, we know that any eigenfunction of a convex domain is in the Sobolev space $H^2(\Om)$, see \cite{Gri}. Let us assume, for a contradiction, that $\lambda_2(\Omega)$ is not simple, then it is double because $\Omega$ is a convex domain in the plane, see \cite{lin}. Let us recall the result of derivability of eigenvalues in the multiple case (see \cite{cox3} or \cite{Rou}). Assume that the domain $\Omega$ is modified by a regular vector field $x\mapsto x+tV(x)$. We will denote by $\Om_t$ the image of $\Om$ by this transformation. Of course, $\Omega_t$ may be not convex but we have actually no convexity constraint (since convexity come for free) and this has no consequence on the differentiability of $t\mapsto \lambda_2(\Omega_t)$. Let us denote by $u_2,u_3$ two orthonormal eigenfunctions associated to $\lambda_2,\lambda_3$. Then, the first variation of $\lambda_2(\Om_t),\lambda_3(\Om_t)$ are the repeated eigenvalues of the $2\times 2$ matrix
\begin{equation}\label{2.35}
\mathcal{M}=\left(
\begin{array}{cc}
-\int_{\p\Om}\big(\frac{\partial
u_2}{\partial n}\big)^ 2\,V.n\,d\sigma&
-\int_{\p\Om}\big(\frac{\partial u_2}{\partial n}\,\frac{\partial
u_3}{\partial n}\big)\,V.n\,d\sigma\\
\ &\ \\
-\int_{\p\Om}\big(\frac{\partial u_2}{\partial n}\,\frac{\partial
u_3}{\partial n}\big)\,V.n\,d\sigma&
-\int_{\p\Om}\big(\frac{\partial u_3}{\partial n}\big)^2\,V.n\,d\sigma
\end{array}
\right)\;.
\end{equation}
Now, let us introduce the Lagrangian $L(\Om)=\lambda_2(\Om)+\mu P(\Om)$. As we will see in the proof of Theorem \ref{theo1}, the perimeter is differentiable and the derivative is a linear form in $V.n$ supported on $\p\Om$ (see e.g. \cite[Corollary 5.4.16]{HP}). We will denote by $\langle dP_{\p\Om},V.n\rangle$ this derivative. So the Lagrangian $L(\Om_t)$ has a derivative which is the smallest eigenvalue of the matrix $\mathcal{M}+\mu\langle dP_{\p\Om},V.n\rangle I$ where $I$ is the identity matrix. Therefore, to reach a contradiction (with the optimality of $\Om$), it suffices to prove that one can always find a deformation field $V$ such that the smallest eigenvalue of this matrix is negative. Let us consider two points $A$ and $B$ on $\p\Om$ and two small neighborhoods $\gamma_A$ and $\gamma_B$ of these two points of same length, say $2\delta$. Let us choose any regular function $\varphi(s)$ defined on $(-\delta, +\delta)$ (vanishing at the extremities of the interval) and a deformation field $V$ such that
$$V.n=+\varphi\mbox{ on }\gamma_A,\quad V.n=-\varphi\mbox{ on }\gamma_B,
\quad V.n=0\mbox{ elsewhere}\,.$$
Then, the matrix $\mathcal{M}+\mu\langle dP_{\p\Om},V.n\rangle I$ splits into two matrices $\mathcal{M}_A-\mathcal{M}_B$ where $\mathcal{M}_A$ is defined by (and a similar formula for $\mathcal{M}_B$):
\begin{equation}\label{2.36}
\mathcal{M}_A=\left(
\begin{array}{ll}
\langle dP_{\gamma_A},\varphi\rangle-\int_{\gamma_A}\big(\frac{\partial
u_2}{\partial n}\big)^ 2\!\varphi\,d\sigma&
-\int_{\gamma_A}\big(\frac{\partial u_2}{\partial n}\,\frac{\partial
u_3}{\partial n}\big)\,\varphi\,d\sigma\\
\ &\ \\
-\int_{\gamma_A}\big(\frac{\partial u_2}{\partial n}\,\frac{\partial
u_3}{\partial n}\big)\,\varphi\,d\sigma&
\langle dP_{\gamma_A},\varphi\rangle-\int_{\gamma_A}\big(\frac{\partial u_3}{\partial n}\big)^
2\!\varphi\,d\sigma
\end{array}
\right)\;.
\end{equation}
In particular, it is clear that the exchange of $A$ and $B$ replaces the matrix $\mathcal{M}_A-\mathcal{M}_B$ by its opposite. Therefore, the only case where one would be unable to choose two points $A,B$ and a deformation $\varphi$ such that the matrix has a negative eigenvalue is if $\mathcal{M}_A-\mathcal{M}_B$ is identically zero for any $\varphi$. But this implies, in particular
\begin{equation}\label{2.37}
\int_{\gamma_A}\frac{\partial
u_2}{\partial n}\,\frac{\partial
u_3}{\partial n}\,\varphi\,d\sigma
=\int_{\gamma_B}\frac{\partial
u_2}{\partial n}\,\frac{\partial
u_3}{\partial n}\,\varphi\,d\sigma
\end{equation}
and
\begin{equation}\label{2.37.2}
\int_{\gamma_A}\Big[\Big(\frac{\partial
u_2}{\partial n}\Big)^2-\Big(\frac{\partial
u_3}{\partial n}\Big)^2\Big]\,\varphi\,d\sigma
=\int_{\gamma_B}\Big[\Big(\frac{\partial
u_2}{\partial n}\Big)^2-\Big(\frac{\partial
u_3}{\partial n}\Big)^2\Big]\,\varphi\,d\sigma
\end{equation}
for any regular $\varphi$ and any points $A$ and $B$ on $\p\Om$. This implies that the product $\left(\frac{\partial u_2}{\partial n}\,\frac{\partial u_3}{\partial n}\right)^2$ and the difference $\big(\frac{\partial u_2}{\partial n}\big)^2-\big(\frac{\partial u_3}{\partial n}\big)^2$ should be constant a.e. on $\p\Om$. As a consequence $\big(\frac{\partial u_2}{\partial n}\big)^2$ has to be constant. Since the nodal line of the second eigenfunction touches the boundary in two points
(see \cite{Mel} or \cite{Ale}), $\frac{\partial u_2}{\partial n}$ has to change sign. So we get a function belonging to $H^{1/2}(\partial\Om)$ taking values $c$ and $-c$ on sets of positive measure, which is absurd, unless $c=0$. This last issue is impossible by the Holmgren uniqueness theorem.
\end{proof}

We are now in a position to prove the existence and regularity of optimal domains for problem \eqref{2.1}.

\begin{theo}\label{theo1}
There exists a minimizer $\Om$ for problem \eqref{2.1} and $\Om$ is of class $C^\infty$.
\end{theo}

\begin{proof}
To show the existence of a solution we use the direct method of calculus of variations. Let $\Om_n$ be a minimizing sequence that, according to point 1 above, we can assume made by convex sets. Moreover, $\Om_n$ is a bounded sequence because of the perimeter constraint. Therefore, there exists a convex domain $\Om$ and a subsequence still denoted by $\Om_n$ such that:
\begin{itemize}
\item $\Om_n$ converges to $\Om$ for the Hausdorff metric and for the $L^1$ convergence of characteristic functions (see e.g. \cite[Theorem 2.4.10]{HP}); since $\Om_n$ and $\Om$ are convex this implies that $\Om_n\to\Om$ in the $\gamma$-convergence;
\item $P(\Om)\le c$ (because of the lower semicontinuity of the perimeter for the $L^1$ convergence of characteristic functions, see \cite[Proposition 2.3.6]{HP});
\item $\lambda_2(\Om_n)\to\lambda_2(\Om)$ (continuity of the eigenvalues for the $\gamma$-convergence, see \cite[Proposition 2.4.6]{BB} or \cite[Theorem 2.3.17]{H}, see also Section \ref{seccap} below).
\end{itemize}
Therefore, $\Om$ is a solution of problem \eqref{2.1}.

We notice that the limit $\Omega$ is a ``true'' domain (i.e. it is not the empty set); indeed any degenerating sequence, a sequence shrinking to a segment for instance, converges to the empty set, thus the second eigenvalue blows to infinity.

We go on with the proof of regularity, which is classical, see e.g. \cite{CL}. We refer also to \cite{Bri}, \cite{BHP} and \cite{br-la} for similar results in a more complicated context. Let us consider (locally) the boundary of $\p\Om$ as the graph of a (concave) function $h(x)$, with $x\in (-a,a)$. We make a perturbation of $\p\Om$ using a regular function $\psi$ compactly supported in $(-a,a)$, i.e. we look at $\Om_t$ whose boundary is $h(x)+t\psi(x)$. The function $t\mapsto P(\Om_t)$ is differentiable at $t=0$ (see \cite{Giu} or \cite{HP}) and its derivative $dP(\Om,\psi)$ at $t=0$ is given by:
\begin{equation}\label{2.3}
dP(\Om,\psi):=\int_{-a}^{+a}\frac{h'(x)\psi'(x)\,dx}{\sqrt{1+{h'(x)}^2}}\,.
\end{equation}
In the same way, thanks to Lemma \ref{lem1}, the function $t\mapsto\lambda_2(\Om_t)$ is differentiable (see \cite[Theorem 5.7.1]{HP}) and since the second (normalized) eigenfunction $u_2$ belongs to the Sobolev space $H^2(\Om)$ (due to the convexity of $\Om$, see \cite[Theorem 3.2.1.2]{Gri}), its derivative $d\lambda_2(\Om,\psi)$ at $t=0$ is
\begin{equation}\label{2.4}
d\lambda_2(\Om,\psi):=-\int_{-a}^{+a}|\nabla u_2(x,h(x))|^2\psi(x)\,dx.
\end{equation}
The optimality of $\Om$ implies that there exists a Lagrange multiplier $\mu$ such that, for any $\psi\in C_0^\infty(-a,a)$
$$\mu d\lambda_2(\Om,\psi)+dP(\Om,\psi)=0$$
which implies, thanks to \eqref{2.3} and \eqref{2.4}, that $h$ is a solution (in the sense of distributions) of the o.d.e.:
\begin{equation}\label{2.5}
-\bigg(\frac{h'(x)}{\sqrt{1+{h'(x)}^2}}\bigg)'=\mu|\nabla u_2(x,h(x))|^2\,.
\end{equation}
Since $u_2\in H^2(\Om)$, its first derivatives $\frac{\p u_2}{\p x}$ and $\frac{\p u_2}{\p y}$ have a trace on $\p\Om$ which belong to $H^{1/2}(\p\Om)$. Now, the Sobolev embedding in one dimension $H^{1/2}(\p\Om)\hookrightarrow L^p(\p\Om)$ for any $p>1$ shows that $x\mapsto|\nabla u_2(x,h(x))|^2$ is in $L^p(-a,a)$ for any $p>1$. Therefore, according to \eqref{2.5}, the function $h'/\sqrt{1+{h'}^2}$ is in $W^{1,p}(-a,a)$ for any $p>1$ (recall that $h'$ is bounded because $\Om$ is convex), so it belongs to some H\"older space $C^{0,\alpha}([-a,a])$ (for any $\alpha<1$, according to Morrey-Sobolev
embedding). Since $h'$ is bounded, it follows immediately that $h$ belongs to $C^{1,\alpha}([-a,a])$. Now, we come back to the partial differential equation and use an intermediate Schauder regularity result (see \cite{GH} or the remark after Lemma 6.18 in \cite{Gil-Tru}) to claim that if $\p\Om$ is of class $C^{1,\alpha}$, then the eigenfunction $u_2$ is $C^{1,\alpha}(\overline{\Om})$ and $|\nabla u_2|^2$ is $C^{0,\alpha}$. Then, looking again to the o.d.e. \eqref{2.5} and using the same kind of Schauder's regularity result yields that $h\in C^{2,\alpha}$. We iterate the process, thanks to a classical bootstrap argument, to conclude that $h$ is $C^\infty$.
\end{proof}

\begin{rem}\rm
Working harder, it seems possible to prove analyticity of the boundary. It would also give another proof of points 1 and 2 of Theorem \ref{theo3} below.
\end{rem}

\subsection{Qualitative properties}
Since we know that the minimizers are of class $C^\infty$, we can now write rigorously the optimality condition. Under variations of the boundary (replace $\Om$ by $\Om_t=(I +tV)(\Om)$), the shape derivative of the perimeter is given by (see Section \ref{sec2.1} and \cite[Corollary 5.4.16]{HP})
$$dP(\Om;V)=\int_{\p\Om}\mathcal{C}\,V.n\,d\sigma$$
where $\mathcal{C}$ is the curvature of the boundary and $n$ the exterior normal vector. Using the expression of the derivative of the eigenvalue given in \eqref{2.4} (see also \cite[Theorem 5.7.1]{HP}), the proportionality of these two derivatives through some Lagrange multiplier yields the existence of a constant $\mu$ such that
\begin{equation}\label{4.1}
|\nabla u_2|^2=\mu\mathcal{C}\quad\mbox{on }\p\Om\,.
\end{equation}
Setting $X=(x_1,x_2)$, multiplying the equality in \eqref{4.1} by $X.n$ and integrating on $\p\Om$ yields, thanks to Gauss formulae $\int_{\p\Om}\mathcal{C}\,X.n\,d\sigma=P(\Om)$, and a classical application of the Rellich formulae $\int_{\p\Om}
|\nabla u_2|^2\,X.n\,d\sigma=2\lambda_2(\Om)$, the value of the Lagrange multiplier. So, we have proved:

\begin{prop}
Any minimizer $\Om$ satisfies
\begin{equation}\label{4.2}
|\nabla u_2(x)|^2=\frac{2\lambda_2(\Om)}{P(\Om)}\,\mathcal{C}(x)\,,\quad x\in\p\Om
\end{equation}
where $\mathcal{C}(x)$ is the curvature at point $x$.
\end{prop}

As a consequence, we can state some qualitative properties of the optimal domains.

\begin{theo}\label{theo3}
An optimal domain satisfies:
\begin{enumerate}
\item Its boundary does not contain any segment.
\item Its boundary does not contain any arc of circle.
\item Its boundary contains exactly two points where the curvature vanishes.
\end{enumerate}
\end{theo}

\begin{proof}
An easy consequence of Hopf's lemma (applied to each nodal domain) is that the normal derivative of $u_2$ only vanishes
on $\p\Om$ at points where the nodal line hits the boundary. Now, we know (see \cite{Mel} or \cite{Ale}) that there are exactly two such points. Then, the first and third items follow immediately from the ``over-determined'' condition \eqref{4.2}. The second item has already been proved in a similar situation in \cite{He-Ou}. We repeat the proof here for the sake of completeness. Let us assume that $\p\Om$ contains a piece of circle $\gamma$. According to \eqref{4.2}, $\Om$
satisfies the optimality condition
\begin{equation}\label{4.2e}
\frac{\p u_2}{\p n}\,=c\;(\mbox{constant})\quad\mbox{on }\gamma\;.
\end{equation}
We put the origin at the center of the corresponding disk and we introduce the function
$$w(x,y)=x\frac{\p u_2}{\p y} - y\frac{\p u_2}{\p x}.$$
Then, we easily verify that
$$\left\{
\begin{array}{ll}
-\Delta w=\lambda_2 w \mbox{ in }\Omega\\
w=0\mbox{ on }\gamma\\
\frac{\partial w}{\partial n}\,=0\mbox{ on }\gamma.
\end{array}\right.$$
Now we conclude, using Holmgren uniqueness theorem, that $w$ must vanish in a neighborhood of $\gamma$, so in the whole domain by analyticity. Now, it is classical that $w=0$ imply that $u_2$ is radially symmetric in $\Om$. Indeed, in polar coordinates, $w=0$ implies $\frac{\p u}{\p \theta}=0$. Therefore $\Om$ would be a disk which is impossible since it would contradict point 3.
\end{proof}

\section{The $N$-dimensional case}\label{sec3}
\subsection{Preliminaries on capacity and related modes of convergence}\label{seccap}
We will use the notion of {\it capacity} of a subset $E$ of $\R^N$, defined by
$$\cp(E)=\inf\Big\{\int_{\R^N}(|\nabla u|^2+u^2)\,dx\ :\ u\in\U_E\Big\}\,,$$
where $\U_E$ is the set of all functions $u$ of the Sobolev space $H^1(\R^N)$ such that $u\ge1$ almost everywhere in a neighbourhood of $E$. Below we summarize the main properties of the capacity and the related convergences. For further details we refer to \cite{BB} or to \cite{HP}.

If a property $P(x)$ holds for all $x\in E$ except for the elements of a set $Z\subset E$ with $\cp(Z)=0$, then we say that $P(x)$ holds {\it quasi-everywhere} (shortly {\it q.e.}) on $E$. The expression {\it almost everywhere} (shortly {\it a.e.}) refers, as usual, to the Lebesgue measure.

A subset $\Omega$ of $\R^N$ is said to be {\it quasi-open} if for every $\ep>0$ there exists an open subset $\Omega_\ep$ of $\R^N$, such that $\cp(\Omega_\ep{\scriptstyle\Delta}\Omega)<\ep$, where ${\scriptstyle\Delta}$ denotes the symmetric difference of sets. Equivalently, a quasi-open set $\Omega$ can be seen as the set $\{u>0\}$ for some function $u$ belonging to the Sobolev space $H^1(\R^N)$. Note that a Sobolev function is only defined quasi-everywhere, so that a quasi-open set $\Omega$ does not change if modified by a set of capacity zero.

In this section we fix a bounded open subset $D$ of $\R^N$ with a Lipschitz boundary, and we consider the class $\A(D)$ of all quasi-open subsets of $D$. For every $\Omega\in\A(D)$ we denote by $H^1_0(\Omega)$ the space of all functions $u\in H^1_0(D)$ such that $u=0$ {\it q.e.} on $D\setminus\Omega$, endowed with the Hilbert space structure inherited from $H^1_0(D)$. This way $H^1_0(\Omega)$ is a closed subspace of $H^1_0(D)$. If $\Omega$ is open, then the definition above of $H^1_0(\Omega)$ is equivalent to the usual one (see \cite{ah96}). For $\Omega\in\A(D)$ the linear operator $-\Delta$ on $H^1_0(\Omega)$ has discrete spectrum, again denoted by $\lambda_1(\Omega)\le\lambda_2(\Omega)\le\lambda_3(\Omega)\le\cdots.$

For $\Omega\in\A(D)$ we consider the unique weak solution $w_\Omega\in H^1_0(\Omega)$ of the elliptic problem formally written as
\begin{equation}\label{eqw}
\left\{
\begin{array}{ll}
-\Delta w=1&\hbox{in }\Omega\\
w=0&\hbox{on }\partial\Omega.
\end{array}\right.
\end{equation}
Its precise meaning is given via the weak formulation
$$\int_D\nabla w\nabla\phi\,dx=\int_D\phi\,dx\qquad\forall\phi\in H^1_0(\Omega).$$
Here, and in the sequel, for every quasi-open set of finite measure we denote by $R_\om: L^2(\R^N) \lra L^2(\R^N)$ the
operator defined by $R_\om (f)=u$, where $u$ solves equation \eqref{eqw} with the right-hand side $f$, so $w=R_\om(1)$.
Now, we introduce two useful convergences for sequences of quasi-open sets.

\begin{defi}\label{gamma}
A sequence $(\Omega_n)$ of quasi-open sets is said to $\gamma$-converge to a quasi-open set $\Omega$ if $w_{\Omega_n}\to
w_\Omega$ in $L^2(\R^N)$.
\end{defi}

\noindent The following facts about $\gamma$-convergence are known (see \cite{BB}).

(i)\ The class $\A(D)$, endowed with the $\gamma$-convergence, is a metrizable and separable space, but it is not compact.

(ii)\ The $\gamma$-compactification of $\A(D)$ can be fully characterized as the class of all {\it capacitary measures} on $D$, that are Borel nonnegative measures, possibly $+\infty$ valued, that vanish on all sets of capacity zero.

(iii)\ For every integer $k$ the map $\Omega\rightarrow\lambda_k(\Omega)$ is a map which is continuous for the $\gamma$-convergence.

To overcome the lack of compactness of the $\gamma$-convergence, it is convenient to introduce a weaker convergence.

\begin{defi}\label{wgamma}
A sequence $(\Omega_n)$ of quasi-open sets is said to $w\gamma$-converge to a quasi-open set $\Omega$ if $w_{\Omega_n}\to w$ in $L^2(\R^N)$, and $\Omega=\{w>0\}$.
\end{defi}

\noindent The main facts about $w\gamma$-convergence are the following (see \cite{BB}).

(i)\ The $w\gamma$-convergence is compact on the class $\A(D)$.

(ii)\ The $w\gamma$-convergence is weaker that the $\gamma$-convergence.

(iii)\ Every functional $F(\Omega)$ which is lower semicontinuous for $\gamma$-conver\-gen\-ce, and decreasing for set inclusion, is lower semicontinuous for $w\gamma$-convergence too. In particular, for every integer $k$, the mapping $\Omega\mapsto\lambda_k(\Omega)$ is $w\gamma$-lower semicontinuous.

(iv)\ The Lebesgue measure $\Omega\mapsto|\Omega|$ is a $w\gamma$-lower semicontinuous map.

This last property can be generalized by the following.

\begin{prop}\label{intf}
Let $f\in L^1(D)$ be a nonnegative function. Then the mapping $\Omega\mapsto\int_\Omega f\,dx$ is $w\gamma$-lower semicontinuous on $\A(D)$.
\end{prop}

\begin{proof}
Let $(\Omega_n)$ be a sequence in $\A(D)$ that $w\gamma$-converges to some $\Omega\in\A(D)$; this means that $w_{\Omega_n}\to w$ in $L^2(\R^N)$ and that $\Omega=\{w>0\}$. Passing to a subsequence we may assume that $w_{\Omega_n}\to w$ a.e. on $D$. Suppose $x\in\Omega$ is a point where $w_{\Omega_n}(x)\to w(x)$. Then $w(x)>0$, and for $n$ large enough we have that $w_{\Omega_n}(x)>0$. Hence $x\in\Omega_n$. So we have shown that
$$1_\Omega(x)\le\liminf_{n\to+\infty}1_{\Omega_n}(x)\qquad\hbox{for {\it a.e.} }x\in D.$$
Fatou's lemma now completes the proof.
\end{proof}

The link between $w\gamma$-convergence and $L^1$-convergence is given by the following.

\begin{prop}\label{link}
Let $(A_n)$ be a sequence of quasi-open sets which $w\gamma$-converges to a quasi-open set $A$, and assume that there exist measurable sets $\Omega_n$ such that $A_n\subset\Omega_n$, and that $(\Omega_n)$ converges in $L^1$ to a measurable set $\Omega$. Then we have $|A\setminus\Omega|=0$.
\end{prop}

\begin{proof}
By applying Proposition \ref{intf} with $f=1_{D\setminus\Omega}$ we obtain
$$|A\setminus\Omega|\le\liminf_{n\to\infty}|A_n\setminus\Omega|=
\liminf_{n\to\infty}|A_n\setminus\Omega_n|=0,$$
which concludes the proof.
\end{proof}

\subsection{A general existence result}\label{subsec3}

In this section we consider general shape optimization problems of the form
\begin{equation}\label{shopt}
\inf\big\{F(\Omega)\ :\ \Omega\subset D,\ P(\Omega)\le L\big\},
\end{equation}
where $D$ is a given bounded open subset of $\R^N$ with a Lipschitz boundary, and $L$ is a given positive real number. Finally, the cost function $F$ is a map defined on the class $\A(D)$ of admissible domains.

We assume that:
\begin{equation}\label{hypof}
\left\{\begin{array}{ll}
\hbox{$F$ is $\gamma$-lower semicontinuous on $\A(D)$;}\\
\hbox{$F$ is decreasing with respect to set inclusion.}
\end{array}\right.
\end{equation}
The functional $F$ is then $w\gamma$-lower semicontinuous by Section \ref{seccap}. Some interesting examples of functionals $F$ satisfying \eqref{hypof} are listed below.

(i) $F(\Omega)=\Phi\big(\lambda(\Omega)\big)$, where $\lambda(\Omega)$ denotes the spectrum of the Dirichlet Laplacian in $\Omega$, that is the sequence $\big(\lambda_k(\Omega)\big)$ of the Dirichlet eigenvalues, and the function $\Phi:\Rbar^\N\to\Rbar$ is lower semicontinuous and nondecreasing, in the sense that
$$\begin{array}{ll}
\lambda_k^n\to\lambda_k\quad \forall k\in\N\ \ \Rightarrow\ \
\Phi(\lambda)
\le\dis\liminf_{n\to\infty}\Phi(\lambda^n)\\
\lambda_k\le\mu_k\quad \forall k\in\N\ \ \Rightarrow\ \
\Phi(\lambda)\le\Phi(\mu).
\end{array}$$

(ii) $F(\Omega)=\cp(D\setminus\Omega)$, where $\cp$ denotes the capacity defined in Section \ref{seccap}.

(iii) $F(\Omega)=\int_D g\big(x,u_\Omega(x)\big)\,dx$, where $g(x,\cdot)$ is lower semicontinuous and decreasing on $\R$ for a.e. $x\in D$, and $u_\Omega$ denotes the solution of
$$\left\{\begin{array}{ll}
-\Delta u=f&\hbox{ in }\Omega\\
u=0&\hbox{ on }\partial\Omega\\
\end{array}\right.$$
with $f\in H^{-1}(D)$ and $f\ge0$.

In order to treat the variational problem \eqref{shopt} it is convenient to extend the definition of the functional $F$ also for measurable sets, where the notion of perimeter has a natural extension. If $M\subset\R^N$ is a measurable set of finite measure, we define
\begin{equation}\label{bh02}
\widehat{H^1_0}(M):=\{u\in H^1(\R^N)\ :\ u=0\mbox{ a.e. on }\R^N\sm M\}.
\end{equation}
For an arbitrary open set $\Om$, this definition does not coincide with the usual definition of $H^1_0(\Om)$. Nevertheless, we point out that it is not restrictive to consider this definition since, for every measurable set $M\subset\R^N$, there exists a uniquely defined quasi open set $\om$ (see for instance \cite{asbu03}) such that
$$H^1_0 (\om)= \widehat{H^1_0}(M),$$
while for a smooth set $\Om$, (e.g. Lipschitz, see \cite[Lemma 3.2.15]{HP}) the definition of the spaces $H^1_0(\Om)$ and $\widehat{H^1_0}(\Om)$ coincide.

With this notion of generalized Sobolev space, one can define 
$$\widehat F(M):=F(\omega).$$
For a nonsmooth open set $\Om$ (say with a crack), we have that $\widehat{H^1_0}(\Om)$ strictly contains $H^1_0(\Om)$, which may lead to the idea that $l_m<l_o$,
where
\begin{equation}\label{lmlo}
\begin{array}{ll}
l_m=\inf\big\{\widehat F(M)\ :\ M\subset D,\ P(M)\le L\big\}\\
l_o=\inf\big\{F(\Omega)\ :\ \Omega\subset D,\ P(\Om)\le L\big\}.
\end{array}
\end{equation} 

In practice, when solving \eqref{shopt}, the minimizing sequence will not develop cracks, precisely because by erasing a crack the generalized perimeter is unchanged and the functional decreases as a consequence of its monotonicity (it may remain constant if the crack coincide with some part of the nodal line of $lambda_2$).

Let us notice that for every measurable set $M$
\begin{equation}\label{exten}
\widehat F(M)=\inf\big\{F(A)\ :\ A\subset M\,a.e.,\ A\in\A(D)\big\}.
\end{equation}

\begin{theo}\label{existence}
There exists a finite perimeter set $M^*$ which solves the variational problem
\begin{equation}\label{shoptm}
\inf\big\{\widehat F(M)\ :\ M\subset D,\ P(M)\le L\big\}.
\end{equation}
\end{theo}

\begin{proof}
Let $(M_n)$ be a minimizing sequence for problem \eqref{shoptm}. Since $P(M_n)\le L$ we may extract a subsequence (still denoted by $(M_n)$) that converges in $L^1$ to a set $M^*$ with $P(M^*)\le L$.

There are quasi-open sets $\omega_n\subset M_n$ a.e. such that
$$F(\omega_n)=\widehat F(M_n).$$
By the compactness of $w\gamma$-convergence we may assume that $(\om_n)$ is $w\gamma$-converging to some quasi-open set $\om$, and by Proposition \ref{link} we have $|\om \setminus M^*|=0$. Therefore, we have that
$$\widehat F(M^*)\le F(\om)\le\liminf_{n\to\infty}F(\om_n)=\liminf_{n\to\infty}\widehat F(M_n).$$
Hence $M^*$ solves the variational problem \eqref{shoptm}.
\end{proof}

The relaxed formulation we have chosen (i.e. to work with measurable sets instead of open sets or quasi-open sets) is not a restriction, provided $F$ verifies the following mild $\g$-continuity property:
\begin{equation}\label{xs12}
\om_n\in\A(D),\ \om_n\subset\om,\ \om_n\sr{\g}{\to}\om\quad\Lra\quad F(\om_n)\to F(\om).
\end{equation}
For instance, all {\it spectral functionals} $F(\Omega)=\Phi\big(\lambda(\Omega)\big)$ seen in (i) above fulfill property \eqref{xs12} provided $\Phi$ is continuous; similarly, integral functionals $F(\Omega)=\int_D g\big(x,u_\Omega(x)\big)\,dx$ seen in (iii) above fulfill property \eqref{xs12} provided $g(x,\cdot)$ is continuous and with quadratic growth.

\begin{theo}\label{equivrel}
Assume that $F$ satisfies \eqref{xs12}. Then, problems \eqref{shopt} and \eqref{shoptm} are equivalent in the sense that $l_o=l_m$, where $l_0$ and $l_m$ are defined in (\ref{lmlo}).
\end{theo}

\begin{proof}
Clearly, $l_m \le l_o$, since for every quasi open set $\Om$ we have $H^1_0(\Om)\subset H^1_0(\om)$ where $H^1_0(\om)=\widehat{H^1_0}(\Om)$.

In order to prove the converse inequality, let $M$ be measurable such that $P(M)\le L$. There exists a quasi-open set $\om\subset M$ a.e. with
$$H^1_0(\om)=\widehat{H^1_0}(M).$$
We point out that the measure of $M\sm\om$ may be strictly positive, and that $P(\om)$ may be strictly greater than $L$.

Following the density result of smooth sets into the family of measurable sets \cite[Theorem 3.42]{afp01}, there exists a sequence
of smooth sets $\Om_n$, such that
$$\left\{\begin{array}{ll}
1_{\Om_n}\to1_M\quad\hbox{in }L^1(\R^N)\\
\Hc^{N-1} (\partial \Om_n) \lra P(M).
\end{array}\right.$$
Unfortunately, it is immediate to observe that this implies only
$$\widehat F(M)\le\liminf_\nif F(\Om_n),$$
while we are seeking precisely the opposite inequality.

We consider the function $w=R_\om (1)$, solution of the elliptic problem \eqref{eqw} in $\omega$, and a sequence $(\rho_k)_k$ of convolution kernels. As in the proof of \cite[Theorem 3.42]{afp01}, by the definition of perimeter and the coarea formula we have
$$\begin{array}{ll}
&\quad\dis P(M)=\int_{\R^N}|\nabla1_M|=\lim_{k\to\infty}\int_{\R^N}|\nabla(1_M*\rho_k)|\,dx\\
&\dis=\lim_{k\to\infty}\int_0^1 P(\{1_M*\rho_k>t\})\,dt
\ge\int_0^1\liminf_{k\to\infty}P(\{1_M*\rho_k>t\})\,dt.
\end{array}$$
For every $t\in(0,1)$ we have
$$\begin{array}{ll}
&\dis|\{1_M*\rho_k>t\}\sm M|\le\frac{1}{t}\int_{\R^N}|1_M*\rho_k - 1_M|\,dx\\
&\dis|M\sm\{1_M*\rho_k>t\}|\le\frac{1}{1-t}\int_{\R^N}|1_M*\rho_k - 1_M|\,dx,
\end{array}$$
so that $1_{\{1_M*\rho_k>t\}}$ converges in $L^1(\R^N)$ to $1_M$ and
$$\liminf_{k\to\infty}P(\{1_M*\rho_k>t\})\ge P(M).$$
The above inequalities imply that for almost every $t\in(0,1)$
$$\liminf_{k\to\infty}P(\{1_M*\rho_k>t\})=P(M).$$
For a subsequence (still denoted using the index $k$) we have
$$\left\{\begin{array}{ll}
1_{\{1_M*\rho_k>t\}}\to1_M\quad\hbox{in }L^1(\R^N)\\
P(\{1_M*\rho_k>t\})\to P(M).
\end{array}\right.$$
We notice that up to a set of zero measure $\{w>0\}\subset M$. We may assume that $w\le1$ (otherwise we consider in the sequel $\frac{w}{\|w\|_\infty}$). Then we get
$$w*\rho_k\le1_M*\rho_k$$
so
$$\{w*\rho_k>t\}\subset\{1_M*\rho_k>t\}$$
and
$$F(\{1_M*\rho_k>t\})\le F(\{w*\rho_k>t\}).$$
Let us prove that
\begin{equation}\label{ineq}
\limsup_{k\to\infty}F(\{w*\rho_k>t\})\le F(\{w>t\}).
\end{equation}
Thanks to \eqref{xs12} and to the monotonicity property, it is enough to prove that $\{w*\rho_k>t\}\cap\{w>t\}$ $\g$-converges to $\{w>t\}$, so to show (see for instance \cite[Chapters 4,5]{BB}) that for every $\vphi\in H^1_0(\{w>t\})$ there exists a sequence $\vphi_k\in H^1_0(\{w*\rho_k>t\})$ such that $\vphi_k\to\vphi$ strongly in $H^1_0$.

Using the density result of \cite{dammu}, it is enough to choose $\vphi\ge0$ such that $\vphi\le(w-t)^+$ and take
$$\vphi_k=\min\{\vphi,\big((w*\rho_k)-t\big)^+\}.$$
Being \eqref{ineq} true for every $t$, from the convergence $F(\{w>t\})\to F(\{w>0\})$ as $t\to0$, by a diagonal procedure we can choose $t_k\to0$ such that
$$\limsup_{k\to\infty}F(\{w*\rho_k>t_k\})\le
F(\{w>0\})=F(\om)=\widehat F(M),$$
and
$$\left\{\begin{array}{ll}
1_{\{1_M*\rho_k>t_k\}}\to1_M\quad\hbox{in }L^1(\R^N)\\
P(\{1_M*\rho_k>t_k\})\to P(M).
\end{array}\right.$$
This proves that $l_o\le l_m$.
\end{proof}

\subsection{The case of the second eigenvalue}
As a corollary of Theorem \ref{existence}, since any eigenvalue of the Laplace operator satisfies \eqref{hypof}, we have:

\begin{theo}\label{existlam}
Let $m=1,2,3,\cdots$, and let $D$ be a bounded open set in $\R^N$ with a Lipschitz boundary. Then the variational problem
\begin{equation}\label{e9}
\inf\big\{\lambda_m(M)\ :\ M\hbox{ measurable set in }D,\ P(M)\le L\big\}
\end{equation}
has a minimizer.
\end{theo}

In this section, we will show that for the second eigenvalue, one can improve the previous Theorem. Actually, we will prove that minimizers exist if $D$ is replaced by all of $\R^N$. The proof relies on a concentration-compactness argument.

\begin{theo}\label{bh07}
Problem
\begin{equation}\label{bh03}
\inf\big\{\lambda_2(M)\ :\ M\mbox{ measurable set in }\R^N,\ P(M)\le L\big\}:=l_m,
\end{equation}
has a solution.
\end{theo}

\begin{proof}
Let $(M_n)$ be a minimizing sequence for problem \eqref{bh03}. We shall use a concentration compactness argument for the resolvent
operators (see \cite[Theorem 2.2]{bujde00}). Let $\om_n$ be the quasi-open sets such that $\om_n\subset M_n$ a.e. and $\lambda_2(\om_n)=\lambda_2(M_n)$. From the classical isoperimetric inequality, the measures of $\om_n$ are uniformly bounded. Consequently, for a subsequence (still denoted using the same index) two situations may occur: compactness and dichotomy.

For a positive Borel measure $\mu$, vanishing on sets with zero capacity, we define (see \cite{BB}) by $R_\mu(f)$ the solution of the elliptic problem
$$-\Delta u+\mu u=f\quad\mbox{in }\R^N,\quad u\in H^1(\R^N)\cap L^2(\mu).$$
If the compactness issue holds, there exists a measure $\mu$ and a sequence of vectors $y_n\in\R^N$ such that the resolvent operators $R_{\om_n+y_n}$ converge strongly in $\Lin(L^2(\R^N))$ to $R_\mu$. Since the perimeters of $M_n$ are uniformly bounded, we can define (up to subsequences) the sets $M^k$ as the limits of $M_n\cap B(0,k)$, and $M=\cup_kM^k$.

Since $w_{\om_n+y_n}$ converges strongly in $L^2(\R^N)$ to $w_\mu$, we get that $\{w_\mu>0\}\subset M$ a.e. so that
$$\lambda_2(M)\le\lambda_2(\{w_\mu>0\})\le\lambda_2(\mu)=\lim_\nif\lambda_2(\om_n).$$
On the other hand
$$P(M)=\lim_{k\to\infty}P(M\cap B(0,k))\le\liminf_{k\to\infty}P(M_k),$$
so that $M$ is a solution to problem \eqref{bh03}.

Let us assume that we are in the dichotomy issue. There exists two sequences of quasi open sets such that
$$\left\{\begin{array}{ll}
\om_n^1\cup\om_n^2\subset\om_n\\
\om_n^1=\om_n\cap B(0,R_n^1),\quad\om_n^2=\om_n\cap(\R^N\sm B(0,R_n^2)),\quad R_n^2-R_n^1\to+\infty\\
\liminf_\nif|\om_n^i|>0,\quad i=1,2\\
R_{\om_n}-R_{\om_n^1}-R_{\om_n^2}\to0\quad\mbox{in }\Lin(L^2(\R^N)).
\end{array}\right.$$
Let us denote
$$M_n^1=M_n\cap B(0,R_n^1),\quad M_n^2=M_n\cap(\R^N\sm B(0,R_n^2)).$$
Since the measures of $M_n$ are uniformly bounded, one can suitably increase $R_n^1$ and decrease $R_n^2$ such that
$$\limsup_\nif P(M_n^1\cup M_n^2)\le L$$
and all other properties of the dichotomy issue remain valid.

We have that $|\lambda_2(\om_n)-\lambda_2(\om_n^1\cup\om_n^2)|\to0$. Since $\om_n^1$ and $\om_n^2$ are disconnected, up to switching the notation $\om_n^1$ into $\om_n^2$, either $\lambda_2(\om_n)$ equals $\lambda_2(\om_n^1)$ or $\lambda_1(\om_n^2)$. The first situation is to be excluded since this implies that $M_n^1$ is a minimizing sequence with perimeter less than or equal to some constant $\alpha<L$, which is absurd. The second situation leads to an optimum consisting on two disjoint balls (this is a consequence of the Faber-Krahn isoperimetric inequality for the first eigenvalue) which is impossible as mentioned in Section \ref{secconnec} below.
\end{proof}

\section{Further remarks and open questions}\label{further}
\subsection{Regularity}
We have proved, in any dimension, the existence of a {\it relaxed} solution, that is a measurable set with finite perimeter. A further step would consist in proving that this minimizer is regular (for example $C^\infty$ as it happens in two dimensions).
It seems to be a difficult issue, in particular because the eigenfunction is not positive, see for example \cite{Bri}, \cite{br-la}. Actually, apart the two-dimensional case, at present we do not even know if optimal domains are open sets.

For a similar problem with perimeter penalization the regularity of optimal domains has been proved in \cite{caf}.

\subsection{Symmetry}
\begin{figure}[h!]
\begin{center}
\scalebox{.4}{\includegraphics{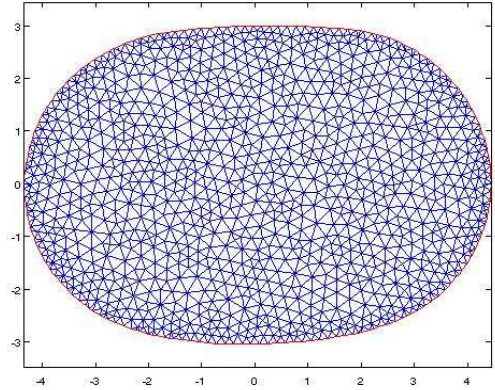}}
\caption{A possible minimizer obtained numerically (by courtesy of Edouard
Oudet)\label{figedouard}}
\end{center}
\end{figure}
Numerical simulations, see Figure \ref{figedouard} show that minimizers in two dimensions should have two axes of symmetry (one of these containing the nodal line), but we were unable to prove it. If one can prove that there is a first axis of symmetry which contains the nodal line, the second axis of symmetry comes easily by Steiner symmetrization.

In higher dimensions, we suspect the minimizer to have a cylindrical symmetry and to be not convex. Indeed, assuming $C^2$ regularity, one can find the same kind of optimality condition as in \eqref{4.2} with the mean curvature instead of the curvature. Since, the gradient of $u_2$ still vanishes where the nodal surface hits the boundary, the mean curvature has to vanish and, according to the cylindrical symmetry, it is along a circle. Therefore, one of the curvatures has to be negative at that point.

\subsection{Higher eigenvalues}
The existence of an optimal domain for higher order eigenvalues under a perimeter constraint is only available when a geometric constraint $\Omega\subset D$ is imposed; we conjecture the existence of an optimal domain also when $D$ is replaced by $\R^N$ but, at present, a proof of this fact is still missing.

\subsection{Connectedness}\label{secconnec}
We believe that optimal domains for problem 
\begin{equation}\label{bh03a}
\inf\big\{\lambda_k(\Om)\ :\ \Om \mbox{ open set in }\R^N,\ {\cal{H}}^{N-1}(\partial\Om)\le L\big\}
\end{equation}
are connected for every $ k$ and every dimension $N$. Actually, for $k=2$, this result is  proved in the forthcoming paper \cite{IVdB}. The idea of the proof consists, first, to show that in the disconnected case, the domain should be the union of two identical balls. Then, a perturbation argument is used: it is shown that the union of two slightly intersecting open balls gives a lower second eigenvalue (keeping the perimeter fixed by dilatation) than two disconnected balls.

\bigskip\noindent
{\bf Acknowledgements.} We wish to thank our colleague Michiel van den Berg for the several stimulating discussions we had on the topic. 

\end{document}